\documentclass[11pt]{article}

\usepackage{amsthm}
\usepackage{amsmath}
\usepackage{amssymb}
\usepackage{amsmath, amsthm, bbm}
\usepackage{framed, fullpage}

\usepackage[backref,colorlinks,citecolor=blue,bookmarks=true]{hyperref}

\newtheorem{theo}{Theorem}

\newtheorem{lemma}{Lemma}[section]
\newtheorem{coro}[lemma]{Corollary}
\newtheorem{definition}[lemma]{Definition}

\newtheorem{claim}[lemma]{Claim}
\newtheorem{fact}[lemma]{Fact}

\newcommand{\F}{\mathcal{F}}

\newcommand{\N}{\mathbb{N}}
\newcommand{\R}{\mathbb{R}}

\newcommand{\ith}{i\text{-th}}

\newcommand{\kth}{k\text{-th}}

\newcommand{\s}[1]{\left\lvert #1 \right\rvert}

\newcommand{\floor}[1]{\left\lfloor{#1}\right\rfloor}
\newcommand{\ceil}[1]{\left\lceil #1 \right\rceil}

\newcommand{\sub}{\subseteq}
\newcommand{\sm}{\setminus}
\renewcommand{\a}{\alpha}

\renewcommand{\d}{\delta}

\newcommand{\e}{\epsilon}

\newcommand{\lm}{\lambda}
\renewcommand{\l}{\ell}
\newcommand{\D}{\Delta}

\DeclareMathOperator{\sep}{sep_E}
\DeclareMathOperator{\sepV}{sep_V}

\newcommand{\E}{\mathbb{E}}

\DeclareMathOperator{\Var}{Var}
\DeclareMathOperator{\poly}{poly}

\DeclareMathOperator{\rank}{rank}

\DeclareMathOperator{\spanV}{span}

\DeclareMathOperator{\UG}{UG}

\date{}

\title{Decomposing a Graph Into Expanding Subgraphs}
\author{
Guy Moshkovitz\thanks{School of Mathematics, Tel Aviv University, Tel Aviv, Israel 69978. Email: {\tt guymosko@tau.ac.il}. Supported in
part by ISF grant 224/11.}
\and Asaf Shapira\thanks{School of Mathematics, Tel Aviv University, Tel Aviv, Israel 69978. Email: {\tt asafico@tau.ac.il}. Supported in part by ISF Grant 224/11 and a Marie-Curie CIG Grant 303320.}
}

\begin{document}

\maketitle

\begin{abstract}

A paradigm that was successfully applied in the study of both pure and algorithmic problems in graph theory can be colloquially summarized as stating that \emph{any graph is close to being the disjoint union of expanders}. Our goal in this paper is to show that in several of the instantiations of the above approach, the quantitative bounds that were obtained are essentially best possible. Three examples of our results are the following:

\begin{itemize}

\item A classical result of Lipton, Rose and Tarjan from 1979 states that if $\F$ is a hereditary family of
graphs and every graph in $\F$ has a vertex separator of size $n/(\log n)^{1+o(1)}$, then every graph in $\F$ has $O(n)$ edges.
We construct a hereditary family of graphs with vertex separators of size $n/(\log n)^{1-o(1)}$ such that not all graphs in the family have $O(n)$ edges.

\item Motivated by the Unique Games Conjecture, Trevisan and Arora-Barak-Steurer 
showed that given a graph $G$, one can remove only $1\%$ of $G$'s edges and thus obtain a graph in which
each connected component has good expansion properties. We show that in both of these decomposition results, the expansion properties they
guarantee are (essentially) best possible even when one is allowed to remove $99\%$ of $G$'s edges.
In particular, our results imply that the eigenspace enumeration approach of Arora-Barak-Steurer
cannot give (even quasi-) polynomial time algorithms for unique games.

\item Sudakov and the second author have recently shown that every graph with average degree $d$ contains an $n$-vertex subgraph
with average degree at least $(1-o(1))d$ and vertex expansion $1/\log^{1+o(1)} n$. We show that one cannot guarantee a better vertex-expansion
even if allowing the average degree to be $O(1)$.

\end{itemize}

The above results are obtained as corollaries of a new family of graphs, which we construct by picking random subgraphs of the hypercube, and
analyze using (simple) arguments from the theory of metric embedding.

\end{abstract}



\section{Introduction}\label{sec:intro}

In recent years a certain paradigm has emerged which roughly says that any graph is close to being a vertex-disjoint union of expanders. Unlike the $\e$-regular partitions of Szemer{\'e}di~\cite{Szemeredi78} and the cut decompositions of Frieze-Kannan~\cite{FriezeKa99} which are relevant only for dense graphs (i.e., with $\Theta(n^2)$ edges), this paradigm is applicable for graphs of arbitrary density. Generally speaking, its usefulness stems from the fact that it allows for a divide-and-conquer approach, in essence reducing a problem on general graphs to the special case of expander graphs.

For algorithmic problems, this paradigm has been useful, for example, in designing approximation algorithms related to the Unique Games Conjecture~\cite{Khot02}, as in the seminal work of Arora-Barak-Steurer~\cite{AroraBaSt10} and of Trevisan~\cite{Trevisan05}. It has also seen applications in property-testing algorithms~\cite{CzumajShSo09,GoldreichRo99}, as well as in data structure design~\cite{PatrascuTh07}.
In graph theory, instantiations of this paradigm include the theorem of Lipton-Tarjan-Rose~\cite{LiptonRoTa79} on vertex separators in hereditary families, 
the results of Linial-Saks~\cite{LinialSa93} and Leighton-Rao~\cite{LeightonRa99} on low-diameter decompositions,
as well as results in the field of graph minors~\cite{KostochkaPy88,ShapiraSu08} and cycle packing~\cite{CFS}.
See Section~\ref{sec:app} for some further discussion.

We note that in all applications of the regularity lemmas~\cite{FriezeKa99,Szemeredi78} one uses the same notion of expansion\footnote{In the setting of the
regularity lemma, expansion is referred to as being $\e$-regular.}.
On the other hand, when dealing with sparse graphs and applying the above-mentioned paradigm, each application calls for a different notion of expansion. Our main goal in this paper is to show that in several of the above-mentioned applications,
the different notions of expansion that were used are quantitatively best possible.

As it turns out, all our results can be deduced from a construction of a new family of graphs (stated in Theorem~\ref{theo:main}), whose main property is that every subgraph of it has a small \emph{edge separator}, which in particular implies small vertex/edge expansion. We next recall some basic definitions.

\begin{definition}[Edge/Vertex Separator]
An \emph{edge separator} in a graph $G=(V,E)$ is a set of edges whose removal disconnects $G$ into two (not necessarily connected) subgraphs, each on at most $\frac23\s{V}$ vertices. We write $\sep(G)$ for the minimum cardinality of an edge separator in $G$.\\
A \emph{vertex separator} is defined in the same way, only now vertices are removed rather than edges. We write $\sepV(G)$ for the minimum cardinality of a vertex separator in $G$
\end{definition}

Note that $\sepV(G)\le\sep(G)$, simply by taking an arbitrary endpoint of each separator edge.
The \emph{girth} of a graph is the minimum length of a cycle in it.
We henceforth write $\log(\cdot)$ for $\log_2(\cdot)$.
Our main technical result in this paper, stated next, implies that there exist graphs with a super-linear number of edges, whose girth is $(\log n)^{1-o(1)}$ and whose $t$-vertex subgraphs all have expansion at most $(\log t)^{-1+o(1)}$.

\begin{theo}\label{theo:main}
For any $n,k$ with $2 \le k \le \frac{1}{648}\log\log n$ there is an $n$-vertex graph $G=G_{n,k}$ satisfying:
\begin{enumerate}
\item $G$ has average degree at least $k$ and maximum degree at most $6k$.
\item $G$ has girth at least $\log n/(6k)^2$.
\item Every $t$-vertex subgraph $H$ of $G$
with $t\ge \log n/(6k)^2$
satisfies
\begin{equation}\label{eq:expansion}
\sepV(H) \le \sep(H)\le \frac{t}{\log t}\cdot(\log\log t)^2 \;.
\end{equation}
\end{enumerate}
\end{theo}

Let us briefly explain why the quantitative estimates in Theorem \ref{theo:main} are (essentially) best possible.
First, observe that for subgraphs $H$ of $G_{n,k}$ in Theorem~\ref{theo:main} on fewer than $\log n/(6k)^2$ vertices, we of course have a much stronger
bound on $\sepV(H)$ and $\sep(H)$ than the one stated in~(\ref{eq:expansion}) for larger subgraphs. To see this, one just has to recall the following well-known fact.

\begin{fact}\label{fact:treeSep}
Every tree $T$ satisfies $\sepV(T)\le 1$.
\end{fact}

So item $(ii)$ in Theorem~\ref{theo:main} implies that subgraphs $H$ of $G_{n,k}$ on fewer than $\log n/(6k)^2$ vertices satisfy
$\sepV(H) \leq 1$ and thus also $\sep(H) \le \D(H) \le \D(G) \, (\le \log\log n)$.\footnote{$\D(\cdot)$ denotes the maximum degree of a graph.}
Note that the bound on $\sep(H)$ cannot be improved by much beyond $\D(G)$ since $G$ contains $K_{1,\D(G)}$ (the star with $\D(G)$ leaves) as a subgraph, and
it is not hard to see that $\sep(K_{1,\Delta}) \geq \Delta/3$. Second, as we discuss in Section~\ref{sec:conc}, any family of graphs
satisfying item $(iii)$ (or variants of it) contains at most $n(\log\log n)^{O(1)}$ edges, so one cannot construct graphs as above with average degree much larger than $\log\log n$. Finally, it follows from Theorem~\ref{theo:SS} below that \emph{any} graph of average degree at least $3$ has a subgraph that is not a tree and does not contain a separator of size much smaller than the one in~(\ref{eq:expansion}), so our estimate on the size of the separator in item $(iii)$ cannot be reduced by more than some $\log\log t$ terms. Actually, the proof of Theorem~\ref{theo:main} shows that in~(\ref{eq:expansion}) we can replace
the $(\log\log t)^2$ term by $(\log\log t)^{1+o(1)}$, but we chose to use the simpler/cleaner expression.

\subsection{Small set expansion}

Next we describe a strengthening of Theorem~\ref{theo:main} which will be important for some of our applications.
Notice that Theorem~\ref{theo:main} gives a graph whose every subgraph has a \emph{large} subset (i.e., consisting of at least $1/3$ of its vertices) that does not expand.
One may thus further ask for a graph that does not contain even \emph{small set expanders}; that is, a graph whose every subgraph has a small non-expanding set (i.e., consisting of $o(1)$-fraction of the vertices).
The notion of small set expanders has recently received much attention in theoretical computer science (see, e.g.,~\cite{AroraBaSt10,Steurer10,SteurerRa10}). For example, it is believed that solving unique games
is intimately connected with the problem of finding small non-expanding sets.

For a graph $H$ and a subset $\emptyset \neq A \sub V(H)$ we write $\partial_H(A)$ for the set of edges of $H$ with precisely one endpoint in $A$, and $\phi_H(A)$ for the \emph{edge expansion} of $A$ in $H$, that is,
$$\phi_H(A) = \frac{\s{\partial_H(A)}}{\s{A}} \;.$$
We have the following ``small set'' counterpart of Theorem~\ref{theo:main}. Note that here we in fact prove the stronger property that one can essentially \emph{partition} the subgraph into small non-expanding subsets.

\begin{theo}\label{theo:SSE}
Let $G=G_{n,k}$ be a graph from Theorem~\ref{theo:main} with $2 \le k \le \frac{1}{648}\log\log\log n$. For every $t$-vertex subgraph $H$ of $G$ with $t\ge \log n/(6k)^2$, and every $1/t \le \mu \le 2/3$, there are $\ceil{1/4\mu}$ mutually disjoint subsets $A_i \sub V(H)$ of size $\mu t/3 \le \s{A_i} \le \mu t$ satisfying
\begin{equation}\label{eq:SSEexpansion}
\phi_H(A_i) \le \frac{\log (1/\mu)}{\log t}\cdot(14\log\log t)^2 \;.
\end{equation}
\end{theo}

Note that in Theorem~\ref{theo:SSE} we can find, in any large subgraph $H$, subsets that simultaneously have ``measure'' $o(1)$ and edge expansion $o(1)$.
For example, by setting $\mu=1/\log t$ in~(\ref{eq:SSEexpansion}), each of the subsets $A \sub V(H)$ produced by the theorem satisfies $\s{A}/t = 1/\log t$ and
$\phi_H(A) \le (1/\log t)(14\log\log t)^3 = 1/(\log t)^{1-o(1)}$.

Similarly to Theorem~\ref{theo:main}, the quantitative estimates in Theorem~\ref{theo:SSE} are (essentially) best possible.
It follows from Theorem~\ref{theo:SS} below that any graph of average degree at least $3$ has a $t$-vertex subgraph whose every subset of measure $\mu$ has edge expansion at least $\frac{\log(1/\mu)}{\log t}(\log\log t)^{-2}$. Thus, our estimate of the edge expansion cannot be reduced by more than some $\log\log t$ terms. Finally, note that while Theorem~\ref{theo:SSE} only deals with subgraphs $H$ on at least $\log n/(6k)^2$ vertices, any smaller subgraph is
a tree (due to item $(ii)$ in Theorem~\ref{theo:main}), implying that it trivially has small subsets with weak expansion.

\subsection{Paper overview}

In Section~\ref{sec:app} we describe the main results of this paper, which are the applications of Theorems~\ref{theo:main} and~\ref{theo:SSE},
showing that many decomposition-type results that were used in different areas of research are essentially best possible.
These include applications in approximation algorithms, property testing, extremal problems in graph theory as well as the nested dissection method.
The proof of Theorem~\ref{theo:main} is given in Section~\ref{sec:proof} (with some technical proofs differed to Section~\ref{sec:add})
and the proof of Theorem~\ref{theo:SSE} is given in Section~\ref{sec:proofSSE}.
Section~\ref{sec:conc} contains some concluding remarks and open problems.
Let us thus give an overview of the proofs of the our theorems.

The main idea of the proof of Theorem~\ref{theo:main} is to start with the Boolean hypercube $Q_d$ and choose an appropriate random subgraph $G$ of it.
Using (simple) probabilistic and combinatorial arguments, we show that this random subgraph $G$ satisfies items $(i)$ and $(ii)$ of Theorem~\ref{theo:main}.
To prove the third item, we first prove that all bounded degree subgraphs of $Q_d$ have small edge expansion. To this end, we use an argument similar to the one used by Linial, London and Rabinovich~\cite{LinialLoRa95} in the setting of metric embedding. 
Then, we proceed to ``boost'' this fact in order to show that every sufficiently large subgraph of $G$ actually has a small edge \emph{separator}, thus establishing item~$(iii)$ of Theorem~\ref{theo:main}.

In order to prove Theorem~\ref{theo:SSE}, we first show that the hereditary nature of item~$(iii)$ of Theorem~\ref{theo:main} can be used to iteratively find smaller and smaller subsets while controlling their expansion.
This enables us to find a single small subset having small expansion.
Next, we remove this subset and look for another small subset, this time having small expansion in the remaining subgraph. We continue in this manner until we get almost a partition of the vertex set. Finally, we show that many of the subsets we obtained have small expansion also in the original subgraph.

\section{Applications of Main Results}\label{sec:app}

\subsection{Edge density of hereditary families with small separators}

A family of graphs is said to be \emph{hereditary} if it is closed under taking induced subgraphs.
The well-known Planar Separator theorem of Lipton-Tarjan~\cite{LiptonTa79} asserts that any $n$-vertex planar graph has a vertex separator of cardinality at most $O(\sqrt{n})$. This influential result lead to many extensions for other hereditary families of graph (such as minor-free families, see~\cite{AlonSeTh94}). The notion of vertex separator in graphs has since found numerous applications in studying both pure and algorithmic problems in graph theory (see~\cite{sepbook} and its references). One of the first applications of separators was given by Lipton, Rose, and Tarjan~\cite{LiptonRoTa79} in their work on the nested dissection method. Intuitively, their result states that the reason planar graphs (and more generally minor-free graphs) have linearly many edges is that they have small separators.

\begin{theo}[\cite{LiptonRoTa79}, Theorem~10]\label{theo:LRT}
For every $\e>0$ there is $C>0$ for which the following holds. Let $\F$ be a hereditary family of graphs such that every $n$-vertex graph $G\in\F$ satisfies $\sepV(G) \le n/(\log n)^{1+\e}$.
Then every $n$-vertex graph in $\F$ has at most $Cn$ edges.
\end{theo}

We mention that Fox and Pach~\cite{FoxPa08} strengthened Theorem~\ref{theo:LRT} by proving that even separators of size
$n/\log n(\log\log n)^{1+\e}$ guarantee\footnote{The result of \cite{FoxPa08} actually shows that having separators of size $n/\log n\log\log n(\log\log\log n)^2$ suffices.} that every graph in $\F$ has $O(n)$ edges.

Using Theorem~\ref{theo:main} we next show that the separation requirement in Theorem~\ref{theo:LRT} cannot be improved much beyond $n/\log n$.

\begin{coro}\label{coro:FoxPach}
There is a hereditary family of graphs $\F$ such that every $n$-vertex graph
$G\in\F$ satisfies $\sepV(G) \le (n/\log n)(\log\log n)^2$, yet there is no $C>0$ such that every $n$-vertex graph in $\F$ has at most $Cn$ edges.
\end{coro}

\begin{proof}
Let $\F$ be the family of graphs defined as follows:
$$
\F=\{G \,:\, G \text{ is an induced subgraph of } G_{n,\log\log n/648} \text{ for some } n\}\;,
$$
where $G_{n,\log\log n/648}$ are the graphs from Theorem~\ref{theo:main}. Note that $\F$ is, by definition, a hereditary family of graphs.
By items $(ii)$ and $(iii)$ of Theorem~\ref{theo:main} (and Fact \ref{fact:treeSep}), every $n$-vertex graph $G\in\F$ satisfies $\sepV(G) \le (n/\log n)(\log\log n)^2$, as required. On the other hand, by item $(i)$ of Theorem~\ref{theo:main}, the family $\F$ contains, for every large enough $n$, an $n$-vertex graph with $\Omega(n \log\log n)$ edges. This completes the proof.
\end{proof}

\subsection{Decomposing a graph into edge expanders}\label{subsec:Trevisan}

Trevisan~\cite{Trevisan05}, in his work on approximating constraints satisfaction problems related to the Unique Games Conjecture,
gave a combinatorial algorithm for approximating unique games.
The main component of the algorithm is a decomposition of any graph into disjoint expanders. Given the graph underlying the unique game input, the algorithm proceeds to solve the problem on each expander separately, which, as Trevisan shows, can be done in polynomial time.
We note that we will revisit this paradigm for solving unique games in  Section~\ref{subsec:ABS}, where we discuss the well-known Arora-Barak-Steurer~\cite{AroraBaSt10} algorithm.
We begin with the formal definitions.

A (linear) \emph{unique game} is a system of linear equations of the form $x_i-x_j=c_{i,j}\pmod{q}$.
Let $\UG_q(\d)$ denote the computational problem of deciding, given a unique game with equations modulo $q$, whether there is an assignment satisfying at least a $(1-\d)$-fraction of the equations or rather every assignment satisfies at most a $\d$-fraction of the equations (we are guaranteed that one of the two cases holds).
The influential Unique Games Conjecture\footnote{The version stated here was shown in~\cite{KhotKiMoOd07} to be equivalent to Khot's original formulation in~\cite{Khot02}.} of Khot~\cite{Khot02} postulates that
$\UG_q(\d)$ is NP-hard for every $0<\d<1/2$ and sufficiently large $q$.
Trevisan's algorithm efficiently solves a weaker version of $\UG_q(\d)$ (and therefore does not refute the Unique Games Conjecture).

A graph $G=(V,E)$ is said to have \emph{edge expansion} $\a$ if every subset $S\sub V$ with $\s{S}\le\s{V}/2$ satisfies $\s{\partial(S)} \ge \a\s{S}$ (recall that the \emph{edge boundary} $\partial(S)$ of $S$ is the set of edges with precisely one endpoint in $S$).
Trevisan proved the following decompositions theorem for arbitrary graphs.
Roughly speaking, it asserts\footnote{We note that prior to~\cite{Trevisan05}, Goldreich and Ron~\cite{GoldreichRo99} have implicitly proved
a result of the same spirit, whose exact quantitative properties are somewhat more complicated to state. See also~\cite{KannanVeVe04}} that, given any graph, one can remove few of the edges in order to disconnect it into disjoint
graphs of high expansion\footnote{Trevisan originally used the notion of conductance; our lower bound applies even for edge expansion.}.
\begin{theo}[\cite{Trevisan05}, Lemma~10]\label{theo:Trevisan}
Let $0<\e\le 1/2$.
From any $n$-vertex graph one can remove at most an $\e$-fraction of the edges to obtain a graph whose every connected component has edge expansion at least $\e/12\log n$.
\end{theo}

Using Theorem~\ref{theo:main} we show in Corollary \ref{coro:Trevisan} below that, already for $\e=1/2$, the expansion guarantee in Theorem~\ref{theo:Trevisan} cannot be improved much beyond $1/\log n$. In fact,
we show that the same conclusion holds even
when $\e=1-o(1)$, that is, even if one is allowed to remove all but a $o(1)$-fraction of the edges of the graph!
To this end, we will first need the following observation regarding the graphs
$G_{n,k}$ from Theorem~\ref{theo:main}.

\begin{claim}\label{claim:many_vertices}
Any subgraph of $G_{n,k}$ with average degree at least $4$ has at least $n^{1/72k^2}$ vertices.
\end{claim}
\begin{proof}
Let $H$ be a $t$-vertex subgraph of $G$ with average degree at least $4$.
Recall the well-known fact that the girth of a $t$-vertex graph of average degree at least $4$ is at most $2\log t$ (in fact, average degree at least $3$ suffices, see~\cite{AlonHoLi02}).
Item $(ii)$ in Theorem~\ref{theo:main} thus implies $\log n/(6k)^2 \le 2\log t$,
completing the proof.
\end{proof}

\begin{coro}\label{coro:Trevisan}
Let $G=G_{n,k}$ be a graph from Theorem~\ref{theo:main}. Any subgraph of $G$ with average degree at least $4$ has a connected component with edge expansion at most $(1/\log n)(15k\log\log n)^2$.
\end{coro}
\begin{proof}
Consider a subgraph of $G$ with average degree at least $4$, and note that it has a connected component $H$ with average degree at least $4$.
Writing $t$ for the number of vertices of $H$, Claim~\ref{claim:many_vertices} implies that $\log t \ge \log n/72k^2$.
It follows from item $(iii)$ in Theorem~\ref{theo:main}
that the edge expansion of $H$ is at most $(t/\log t)(\log\log t)^2\cdot(3/t) \le (1/\log n)(15k\log\log n)^2$.
\end{proof}

By Corollary~\ref{coro:Trevisan}, even after removing a fraction of $1-4/k$ of the edges of $G_{n,k}$, the remaining graph has a connected component with edge expansion at most $O((1/\log n)(k\log\log n)^2)$.
So for example, taking $k=8$ gives a bounded-degree graph such that even after removing up to half of its edges, the remaining graph has a connected component with edge expansion at most $O((1/\log n)(\log\log n)^2)$.
At the other extreme, if $k=\log\log n/648$ then even after removing a $(1-o(1))$-fraction of the edges, the remaining graph has a connected component with edge expansion at most $(1/\log n)(\log\log n)^4$.

\subsection{Threshold-rank decomposition}\label{subsec:ABS}

Recall that in Section~\ref{subsec:Trevisan} we discussed a paradigm for solving unique games by decomposing any graph into expanding subgraphs.
The same paradigm was used in the breakthrough paper of Arora, Barak and Steurer~\cite{AroraBaSt10} which gave the first algorithm to come close to refuting the Unique Games Conjecture, solving $\UG_q(\d)$ in subexponential time $\exp(q\cdot n^{\d^{1/3}})$ (here $n$ is the number of variables).
The algorithm in~\cite{AroraBaSt10} employed the \emph{eigenspace enumeration method} where, unlike Trevisan's decomposition, the pseudorandom property which the connected components must satisfy depends on the entire spectrum of their adjacency matrix.
As in Section~\ref{subsec:Trevisan}, we prove that one cannot substantially improve the obtained decomposition.
Formal definitions follow.

For a $d$-regular graph $G$, denote by $\rank_{\tau}(G)$ the number (with multiplicities) of eigenvalues $\lm$ of the
adjacency matrix of $G$ satisfying $\s{\lm} > \tau d$.
Let $R_n(\eta,\e)$ denote the minimum integer such that for any $n$-vertex 
graph $G$, one can remove at most an $\e$-fraction of its edges so that each connected component $H$ of the new graph satisfies $\rank_{1-\eta}(H^*) \le R_n(\eta,\e)$, where $H^*$ is obtained from $H$ by adding self-loops to make it $\D(G)$-regular.
The main result of \cite{AroraBaSt10} can be summarized as follows.
\begin{theo}[\cite{Steurer10}, Theorem~5.6]\label{theo:ABS}
There are absolute constants $c,c' > 0$ such that:
\begin{enumerate}
\item For every $0<\eta,\e \le 1$, 
$$R_n(\eta,\e) \le n^{c(\eta/\e^2)^{1/3}} \;.$$
Moreover, a decomposition achieving this bound can be found in polynomial time.\footnote{So for example, for every $d$-regular $n$-vertex graph $G$, one can efficiently 
remove at most (say) $1\%$ of the edges of $G$ so that for each connected component in the resulting graph, its ``regularized'' adjacency matrix has at most $n^{O(\eta^{1/3})}$ eigenvalues larger than $(1-\eta)d$.}
\item $\UG_q(\d)$ can be solved in time $\exp(\,q \cdot R_n(O(\d),c')\,)$.
\end{enumerate}
\end{theo}

The first step in the algorithm of~\cite{AroraBaSt10} (item $(i)$ in Theorem~\ref{theo:ABS} which, as the authors themselves write, is the main component of the algorithm) is the efficient decomposition of any graph, by removing few of its edges, into vertex-disjoint subgraphs of small threshold rank. This decomposition is applied on the graph underlying the input unique game, and then the problem is solved on each connected component.
The second step (item $(ii)$ in Theorem~\ref{theo:ABS}) follows by applying, for each connected component, a brute force enumeration over the eigensapces corresponding to its large eigenvalues.
The second step is inspired by works of Kolla and Tulsiani~\cite{AroraKhKoStTuVi08} and Kolla~\cite{Kolla10}, who previously
used eigenspace enumeration to solve unique games on Cayley graphs.
Combining the two items of Theorem~\ref{theo:ABS} shows that the algorithm of~\cite{AroraBaSt10} runs in time $\exp(q\cdot n^{O(\d^{1/3})})$.

Note that in order to obtain a faster algorithm using the eigenspace enumeration method, one must improve the bound on $R_n(\eta,\e)$ in Theorem~\ref{theo:ABS}.
Up to the present work it was conceivable\footnote{The main tool used in~\cite{AroraBaSt10} in order to prove Theorem~\ref{theo:ABS} is a higher-order Cheeger-type inequality. It was recently shown in~\cite{BarakGoHaMeRaSt12} that this inequality is essentially best possible, but to the best of our knowledge, this does not rule out the possibility of improving Theorem~\ref{theo:ABS} by other means.} that for example $R_n(\eta,\e) \le (\log n)^{\poly(\eta)}$, say for $\e \ge \e_0$. Such a decomposition would have implied a quasipolynomial time algorithm for $\UG_q(\d)$, practically refuting\footnote{Say, under the
so-called {\em Exponential Time Hypothesis}.} the Unique Games Conjecture.

As we prove below, the bound on $R_n(\eta,\e)$ in Theorem~\ref{theo:ABS} is essentially tight. Namely, we show that one can derive from Theorem~\ref{theo:SSE} a nearly polynomial lower bound on $R_n(\eta,\e)$.
This essentially eliminates the possibility of disproving the Unique Games Conjecture using the eigenspace enumeration approach of~\cite{AroraBaSt10}. See Section~\ref{sec:conc} for more details on this.
We note that our lower bound applies even when $\e=1-o(1)$, that is, even when one is allowed to remove all but a $o(1)$-fraction of the graph's edges.
Naturally, the bound holds regardless of the running time of the decomposition algorithm.

\begin{coro}\label{coro:ABS-R}
For every $n \ge n_0$, $0<\eta \le 1$, and $\e \le 1-3000/\log\log\log n$, we have
$$R_n(\eta,\e) \ge n^{\eta/(\log\log n)^3} \;.$$
\end{coro}

For the proof of Corollary~\ref{coro:ABS-R} we will need the ``easy'' direction of the higher-order Cheeger inequality
that first appeared in~\cite{LeeOGTr12}. For completeness, we give a short and simple proof in Section~\ref{sec:add}.

\begin{lemma}\label{lemma:Cheeger}
Let $G=(V,E)$ be a $d$-regular (multi-)\footnote{By this we mean that we allow self-loops.}graph, and let $\lm_1\ge\lm_2\ge\cdots\ge\lm_{\s{V}}$ be the eigenvalues of the adjacency matrix of $G$. Then for every $k$,
$$\frac{d-\lm_k}{2} \le \min_{S_1,\ldots,S_k} \max_{1\le i\le k} \phi_G(S_i) \;,$$
where the minimum is over all collections of $k$ mutually disjoint, non-empty  subsets of $V$.
\end{lemma}

\begin{proof}[Proof of Corollary~\ref{coro:ABS-R}]
Let $G=G_{n,k}$ be a graph from Theorem~\ref{theo:main}.
We will show that any subgraph of $G$ with average degree at least $4$ has a connected component $H$ satisfying
\begin{equation}\label{eq:tr-component}
\rank_{1-\eta}(H^*) \ge \frac{1}{4} n^{\eta/k(200\log\log n)^2} \;.
\end{equation}
Setting $k=\log\log\log n/648$, this would mean that removing an $\e$-fraction of the edges of $G$, even with $\e$ as large as $1-4/k \ge 1-3000/\log\log\log n$,  leaves a connected component $H$ that, for every $n\ge n_0$, satisfies $\rank_{1-\eta}(H^*) \ge n^{\eta/(\log\log n)^3}$. This would imply the desired lower bound on $R_n(\eta,\e)$.

Consider a subgraph of $G$ with average degree at least $4$ and note that it has a connected component $H$ with average degree at least $4$.
Writing $t$ for the number of vertices of $H$, Claim~\ref{claim:many_vertices} implies that $t \ge n^{1/72k^2}$ ($\ge \log n/(6k)^2$).
Apply Theorem~\ref{theo:SSE} on $H$ with $\mu=1/t^{\eta k/(20\log\log t)^2}$.
Note that, as required by Theorem~\ref{theo:SSE}, we may assume $\mu \ge 1/t$ (as one may verify, recalling $\eta \le 1$) and $\mu \le 2/3$ (as otherwise $n^{\eta/k(200\log\log n)^2}\le 1/\mu \le 3/2$, so~(\ref{eq:tr-component}) trivially holds).
We thus obtain
$$b:= \frac{1}{4\mu} = \frac14 t^{\eta k/(20\log\log t)^2} \ge \frac{1}{4} n^{\eta/k(200\log\log n)^2}$$
mutually disjoint and non-empty subsets of $H$ that each have edge expansion in $H$ at most $(\log(1/\mu)/\log t)(14\log\log t)^2 < \eta k/2$. Let $\lm_1\ge\lm_2\ge\cdots\ge\lm_{t}$ be the eigenvalues of the adjacency matrix of $H^*$ (recall that $H^*$ is obtained by adding to $H$ sufficiently many self-loops so as to make it $\D(G)$-regular). Note that adding self-loops does not alter the expansion of any subset. Therefore, Lemma~\ref{lemma:Cheeger} implies that $\D(G)-\lm_b < \eta k$. It follows that for every $1 \le i \le b$ we have $\lm_i \ge \lm_b > \D(G)-\eta k \ge (1-\eta)\D(G)$. This shows that $\rank_{1-\eta}(H^*) \ge b$, as required in~(\ref{eq:tr-component}), and therefore completes the proof.
\end{proof}

\subsection{Finding a single vertex-expanding subgraph}\label{subsec:SS}

For a graph $G$, a subset $S \sub V(G)$ is said to have \emph{vertex expansion} $\a$ if $|N(S)| = \a\s{S}$, where $N(S)$ denotes the set of vertices outside of $S$ that have a neighbor in $S$.
Motivated by certain extremal problems related to graph minors, it was shown in~\cite{ShapiraSu08} that every graph $G$ contains a subgraph $H$, such that $H$ has good vertex expansion properties and almost the same average degree as $G$.

\begin{theo}[\cite{ShapiraSu08}, Lemma~1.2]\label{theo:SS}
Let $0<\e\le 2^{-8}$. Every graph $G$ of average degree $k$ contains a $t$-vertex subgraph $H$ with average degree at least $(1-\e)k$ such that every subset of $V(H)$ of size $\mu t$ with $1/t \le \mu \le 1/2$ has vertex expansion in $H$ at least
$\e\cdot\log(1/\mu)/\log t(2\log\log t)^2$.
\end{theo}

Theorem~\ref{theo:SSE}, together with item~(ii) of Theorem~\ref{theo:main} and Fact~\ref{fact:treeSep}, yields the following corollary, which shows that the expansion guarantee in Theorem~\ref{theo:SS} above is essentially best possible.

\begin{coro}\label{coro:SS}
Let $G$ be a graph from Theorem~\ref{theo:SSE}, and let $H$ be a $t$-vertex subgraph of $G$. Then for every $1/t \le \mu \le 2/3$,
there is a subset of $V(H)$ of size in $[\mu t/3,\mu t]$ whose vertex expansion in $H$ is at most $(\log (1/\mu)/\log t)\cdot(14\log\log t)^2$.
\end{coro}

\subsection{Hyperfinite families of graphs}

A graph is said to be \emph{$(\e,q)$-hyperfinite} if one can remove an $\e$-fraction of its edges and thus decompose it into connected components of size at most $q$ each. A family of graphs is said to be \emph{hyperfinite} if there is a function $q$ such that for every $\e >0$, every graph in the family is $(\e,q(\e))$-hyperfinite. Hyperfinite families of graphs have been extensively studied in recent years, mainly because of their role in the
theory of graph limits of sparse graphs (see \cite{L}).
Motivated by certain questions related to the design of property-testing algorithms, it was shown in~\cite{CzumajShSo09} that a hereditary family of graphs in which every graph has a small edge separator must be hyperfinite. More precisely, the following holds.

\begin{theo}[\cite{CzumajShSo09}, Corollary~3.2]\label{theo:CSS}
Let $\F$ be a hereditary family of graphs such that every $n$-vertex graph $G \in \F$ satisfies $\sep(G) \leq n/\log n(\log\log n)^2$.
Then, $\F$ is hyperfinite.
\end{theo}

Using Theorem~\ref{theo:main} we show that the edge-separation requirement in Theorem~\ref{theo:CSS} cannot be improved much beyond $n/\log n$.

\begin{coro}\label{coro:CSS}
The graphs $G_{n,8}$ from Theorem~\ref{theo:main} are not $(\frac12, n^{1/4608})$-hyperfinite. In particular, setting
$$\F_8=\{G \,:\, G \text{ is an induced subgraph of } G_{n,8} \text{ for some } n\}$$ we get a hereditary family of graphs that is not hyperfinite, despite the fact that every $n$-vertex graph $G \in \F_8$ satisfies $\sep(G) \le (n/\log n)(\log\log n)^2+48$.
\end{coro}
\begin{proof}
To see that the first assertion holds, note that after removing at most half of the edges of $G_{n,8}$, we obtain a graph with average degree at least $4$. This graph has a connected component $H$ of average degree at least $4$. By Claim~\ref{claim:many_vertices}, the number of vertices in $H$ is least $n^{1/4608}$.
As to the second assertion of the corollary, note that the first assertion clearly means that $\F_8$ is not hyperfinite.
Also, note that the fact that $\sep(G) \leq (n/\log n)(\log\log n)^2+48$ holds for every $G \in \F_8$ follows from items $(i)$,$(ii)$ and $(iii)$ of Theorem~\ref{theo:main} together with Fact~\ref{fact:treeSep}.
\end{proof}

\section{Proof of Theorem \ref{theo:main}}\label{sec:proof}

As mentioned earlier, in order to construct the graphs $G_{n,k}$ of Theorem~\ref{theo:main} we pick a suitable random subgraph of the hypercube.
The $d$-cube, denoted $Q_d$, is the graph with vertex set $\{0,1\}^d$ where two vertices
are adjacent if their corresponding vectors
differ in exactly one coordinate.
Note that $Q_d$ is a $d$-regular graph on $2^d$ vertices.
The proof will have three main steps, appearing in the following three subsections. In the first subsection we define the graphs $G_{n,k}$ and prove that they satisfy the first two properties of Theorem~\ref{theo:main}. In the second step we prove that large subgraphs of the
cube do not have good expansion properties. Finally, in the last subsection we combine the results from the first two subsections to prove that the graphs $G_{n,k}$ also satisfy the third property in Theorem~\ref{theo:main}, thus completing its proof.

\subsection{Construction}

Our goal in this subsection is to prove the following lemma.

\begin{lemma}\label{lemma:construction}
For any $d\in\N$ and $2 \leq k \leq d$ there is a subgraph $Q=Q_{d,k}$ of the $d$-cube, on all $2^d$ vertices,
such that:
\begin{enumerate}
\item The average degree of $Q$ is at least $k$.
\item The maximum degree of $Q$ is at most $3k$.
\item The girth of $Q$ is at least $d/(3k)^2$.
\end{enumerate}
\end{lemma}

For the proof of the above lemma we will need the following upper bound on the number of cycles of a given length in the cube.

\begin{claim}\label{claim:cycles}
The number of cycles of length $2\l$ in the $d$-cube is at most $2^d (d\l)^{\l}$.
\end{claim}
\begin{proof}
We will show that the number of closed walks of length $2\l$, starting and ending at a given vertex, is at most $(d\l)^{\l}$; this would immediately imply the stated bound.
We claim that each such closed walk corresponds to a sequence $(x_1,\ldots,x_{2\l})\in[d]^{2\l}$ with the property that $\s{\{i\in[2\l]\,:\, x_i=t\}}$ is even for every $t\in [d]$. To see this, recall that each vertex in the graph corresponds to an element of $\{0,1\}^d$. The claim follows by considering, for each edge along a closed walk beginning (and ending) with a given vertex, the unique index in which the bit ``flips''.
So all that is left is to bound from above the number of sequences as above.
Now, observe that all these sequences can be generated by first partitioning the set of indices $\{1,\ldots,2\l\}$ into $\l$ pairs and then
assigning to each pair a value from $\{1,\ldots,d\}$ (of course, this process will generate some sequences several times). Since the number of ways one can pair the elements of $\{1,\ldots,2\ell\}$ is given by
$$(2\l-1)!!=(2\l-1)(2\l-3)\cdots 3\cdot 1\le 2^\l \l!/2 \le 2^\l (\l/2)^{\l} = \l^\l\;,$$
we get that the number of sequences is bounded from above by $d^\ell\cdot\l^\l$.
\end{proof}

To prove the existence of a graph as in Lemma~\ref{lemma:construction} we use the so-called probabilistic ``deletion method''.

\begin{proof}[Proof of Lemma~\ref{lemma:construction}]
Denote by $G'$ the random subgraph of the $d$-cube where each edge is independently retained with probability $p=3k/d$.
Let the random variable $X$ count the number of edges of $G'$. Then
\begin{equation}\label{eq:cons1}
\E[X]=p\cdot d2^{d-1} = 3k \cdot 2^{d-1} \;.
\end{equation}
Set $L=d/9k^2$ and let the random variable $Y$ count the number of cycles of length at most $L$ in $G'$. By Claim~\ref{claim:cycles},
\begin{equation}\label{eq:cons2}
\E[Y] \le \sum_{\l=2}^{L/2} p^{2\l} 2^d(d\l)^\l =
2^d\sum_{\l=2}^{L/2} (p^2 d\l)^\l =
2^d \sum_{\l=2}^{L/2} (\l/L)^{\l} \le
2^d \sum_{\l=2}^{L/2} (1/2)^{\l} \le 2^{d-1} \;.
\end{equation}
Let the random variable $Z$ count the total number of ``excess'' edges in $G'$, that is,
$$Z=\sum_{v \,:\,\deg(v)>3k} (\deg(v)-3k)\;.$$
For each vertex $v$, the random variable $\deg(v)$ follows the binomial distribution $B(d,3k/d)$, so $\E[\deg(v)]=3k$.
Note that $\sum_{v} \s{\deg(v)-3k} = \sum_v(3k-\deg(v)) + 2Z$, which means $\E[Z] = \frac12\sum_{v} \E\s{\deg(v)-3k}$.
By J\'{e}nsen's inequality,
$$(\E\s{\deg(v)-3k})^2 \le \E[(\deg(v)-3k)^2] = \Var[\deg(v)] \le 3k \;,$$
implying that
\begin{equation}\label{eq:cons3}
\E[Z] = \frac12 \sum_v \E\s{\deg(v)-3k} \le \sqrt{3k}\cdot 2^{d-1} \;.
\end{equation}
Combining~(\ref{eq:cons1}),(\ref{eq:cons2}), and (\ref{eq:cons3}) we get
\begin{equation}\label{eq:Expectation}
\E[X-Y-Z] \ge (3k - 1 - \sqrt{3k}) \cdot 2^{d-1} \ge k \cdot 2^{d-1} \;,
\end{equation}
where the last inequality can be easily checked to hold for any $k\ge 2$.
Let $Q$ be obtained from $G'$ by removing an arbitrary edge from each cycle of length at most $d/9k^2$, as well as removing, for each vertex $v$ with $\deg(v)>3k$, arbitrary $\deg(v)-3k$ adjacent edges.
Clearly, $Q$ satisfies the last two requirements in the statement. Moreover, we have from~(\ref{eq:Expectation}) that the expected average degree of $Q$ is at least $k$.
The existence of a subgraph of the $d$-cube as required immediately follows.
\end{proof}

\subsection{Expansion in the hypercube}

In this subsection we prove that large subgraphs of the $d$-cube are not good edge expanders.
Recall that a graph $G=(V,E)$ is said to have edge expansion $\a$ if every subset $S\sub V$ with $\s{S}\le\s{V}/2$ satisfies $\s{\partial(S)} \ge \a\s{S}$ (the edge boundary $\partial(S)$ is the set of edges with exactly one vertex in $S$). Our goal in this subsection is to prove the following lemma.

\begin{lemma}\label{lemma:cubeExp}
Any (not necessarily induced) subgraph of the $d$-cube with $t$ vertices and average degree $r$ has edge expansion at most $2r\log d/\log (t/2)$.
\end{lemma}

We first need the following easy claim, in which the \emph{distance} between two vertices $u,v$ in a graph $G$,
denoted  $\d_G(u,v)$, is the length of a shortest path connecting $u$ and $v$.

\begin{claim}\label{claim:ball}
For every graph $G$ of maximum degree $\D>1$, and every $t$-vertex subset $S\sub V(G)$, the average distance $\sum_{\{u,v\}\in\binom{S}{2}} \d_G(u,v)/\binom{t}{2}$ is at least
$\log (t/2)/(2\log\D)$.
\end{claim}
\begin{proof}
Let $v\in S$.
We claim that there are at least $t/2$ vertices in $S$ of distance at least $\l=\log(t/2)/\log\D$ from $v$. Indeed, the number of vertices of distance at most $\l-1$ from $v$ is at most $\sum_{i=0}^{\floor{\l-1}} \D^i < \D^{\l} \le t/2$. It follows that the average distance is at least $\l/2$.
\end{proof}

We are now ready to prove the main result of this subsection.

\begin{proof}[Proof of Lemma~\ref{lemma:cubeExp}]
Let $H=(S,E)$ be a $t$-vertex subgraph of the $d$-cube $Q_d$ of average degree $r$, and let $\a$ denote the edge expansion of $H$. We need to show that $\a \le 2r\log d/\log (t/2)$. Note that we may assume $\alpha>0$ (i.e., that $H$ is connected), as otherwise there is nothing to prove.
For each vertex $v \in S$, write $(v_1,\ldots,v_d)\in\{0,1\}^d$ for the corresponding binary vector.
Notice $\d_{Q_d}(u,v)= \sum_{i=1}^d \s{u_i-v_i}$.
Observe that
\begin{equation}\label{eq:frac}
\frac{\sum_{\{u,v\}\in E} \d_{Q_d}(u,v)}{\sum_{\{u,v\}\in \binom{S}{2}} \d_{Q_d}(u,v)} =
\frac{\sum_{i=1}^d  \sum_{\{u,v\}\in E} \s{u_i-v_i}}{\sum_{i=1}^d \sum_{\{u,v\}\in \binom{S}{2}} \s{u_i-v_i}} =:
\frac{\sum_{i=1}^d x_i}{\sum_{i=1}^d y_i} \ge
\min_{i:\,y_i\neq 0} \frac{x_i}{y_i} \;,
\end{equation}
where the inequality holds\footnote{Indeed, denoting the minimum by $m$, we have $x_i\ge m y_i$ for every $i$ (even if $y_i=0$), so $(\sum_{i=1}^d x_i)\ge m(\sum_{i=1}^d y_i)$.} for all non-negative reals $x_1,y_1,\ldots,x_d,y_d$.
Let $i\in[d]$ achieve the minimum in~(\ref{eq:frac}) (which is well defined since  $y_i\neq 0$ for some $i$), and set $T=\{v\in S \,:\, v_i=1\}$. Note that since $y_i >0$ we have $0<\s{T}<\s{S}$ and that we can in fact assume that
$0 < \s{T}\le\s{S}/2$ (otherwise we replace $T$ with $S\sm T$).
Noting that in~(\ref{eq:frac}) we have $x_i=\s{\partial{T}}$ and $y_i=\s{T}\s{S\sm T}$, we deduce that
$$
\frac{\sum_{\{u,v\}\in E} \d_{Q_d}(u,v)}{\sum_{\{u,v\}\in \binom{S}{2}} \d_{Q_d}(u,v)} \ge \frac{\s{\partial{T}}}{\s{T}\s{S\sm T}} \ge
\min_{\substack{ S'\sub S\,:\ \\ 0<\s{S'}\le\s{S}/2}} \frac{\s{\partial{S'}}}{\s{S'}\s{S\sm S'}} \ge
\frac{\a}{\s{S}-1} \;.
$$
Therefore,
\begin{equation}\label{eq:avgDist}
\frac{1}{\binom{\s{S}}{2}}\sum_{\{u,v\}\in\binom{S}{2}} \d_{Q_d}(u,v) \le \frac{2}{\a\s{S}}\sum_{\{u,v\}\in E} \d_{Q_d}(u,v) = \frac{2\s{E}}{\a\s{S}} = \frac{r}{\a} \;.
\end{equation}
Applying Claim~\ref{claim:ball} with $G=Q_d$ and the set $S$, the left hand side of~(\ref{eq:avgDist}) is at least
$\log (t/2)/2\log d$.
The desired bound on $\a$ follows.
\end{proof}

\subsection{Putting it all together}

To prove Theorem~\ref{theo:main} we will show that the graphs constructed in Lemma~\ref{lemma:construction} have the property that every subgraph $H$ has a small edge separator. To this end, we start by proving that every subgraph $H$ has a small edge expansion.
In the proof we will consider two cases, depending on the number of vertices of $H$. If $H$ has at most $2^{d^{1/3}}$ vertices, then having
high expansion
implies (see Lemma~\ref{lemma:girthE} below) the existence of a cycle of length $O(d/k^2)$; however, such a short cycle does not exist in the graph. Otherwise, since the maximum degree of the graph is bounded, Lemma~\ref{lemma:cubeExp} implies that $H$ has small edge expansion.
Finally, we boost this ``hereditary'' non-expansion property to construct a small separator in $H$ (see Lemma~\ref{lemma:boost} below).

To execute the above proof strategy we will need the following two lemmas, whose proofs appear in Section~\ref{sec:add}.
Recall that we use $\sep(G)$ to denote the size of the smallest edge separator in $G$.

\begin{lemma}\label{lemma:girthE}
Every connected $n$-vertex graph with edge expansion $\a ~(>0)$ that is not a tree
has girth at most $2(2/\a +1)(\ln(n)+1) + 1$.
In particular, if $\a\le 1$ then the girth is at most $12\ln n/\a$.
\end{lemma}

\begin{lemma}\label{lemma:boost}
Let $G$ be an $n$-vertex graph whose every $t$-vertex subgraph has expansion at most $f(t)$, for every $t\ge n/3$ and where $f:[n/3,n]\to\R$ is decreasing. Then $\sep(G)\le (2/3)n f(n/3)$.
\end{lemma}

We can finally give a proof for the main theorem.

\begin{proof}[Proof of Theorem~\ref{theo:main}]
First, suppose that $n$ is a power of $2$, and write $n=2^d$.
Let $G=Q_{d,2k}$ be a subgraph of the $d$-cube as guaranteed by Lemma~\ref{lemma:construction}.
For the rest of the proof set $t_0=d/(6k)^2$.
The assertion of Lemma~\ref{lemma:construction} implies that $G$ has average degree at least $2k$, maximum degree at most $6k$, and girth at least $t_0$.
So to complete the proof (for $n=2^d$) we only need to establish item $(iii)$ of Theorem~\ref{theo:main}.
We will start by proving that every $t$-vertex subgraph of $G$ on at least $t_0$ vertices has edge expansion at most $\frac13(1/\log t)(\log\log t)^2$.
We will then ``boost'' this fact using Lemma~\ref{lemma:boost} and get that every $t$-vertex subgraph as above in fact has an edge separator of size at most $(t/\log t)(\log\log t)^2$, as required.
Note that the statement of the theorem implies that $d \ge 2^{2\cdot 648}$. One consequence of this fact is that $d \ge (\log d)^6$.
We will also use the fact that the theorem's assumption that $k \leq \log d/648$ together with the previous observation implies that $d^{1/3}/(6k)^2 \geq 12$ and that $18k \leq t_0/(3\log t_0)$. Finally, suppose $H'$ is a subgraph of $G$ and that $H'$ is a forest. As $G$ has maximum degree at most $6k$, we have by Fact~\ref{fact:treeSep} and the above relation between $k$ and $t_0$ that
$\sep(H') \le \D(G) \le 6k \le \frac19 t_0/\log t_0$.
Hence, if $H'$ has $t\geq t_0$ vertices then its edge expansion is
at most $(\frac19 t/\log t)/(\frac13 t) = \frac13 (1/\log t)$.

Let then $H$ be a $t$-vertex subgraph of $G$ with $ t \geq t_0$ vertices, and let us prove that the edge expansion of $H$ is at most $\frac13(1/\log t)(\log\log t)^2$. Since we have already established this fact when $H$ is a forest at the end of the previous paragraph, we assume for the rest of this paragraph that $H$ is not a forest. Suppose first that $t < 2^{d^{1/3}}$.
By Lemma~\ref{lemma:girthE}, if the edge expansion of $H$ is at least $1/\log t$ then its girth is at most $12(\log t)^2 < 12d^{2/3} \le d/(6k)^2$.
However, $G$ does not contain a cycle this short.
Therefore, the edge expansion of $H$ is indeed at most $\frac13(1/\log t)(\log\log t)^2$.
Suppose now that $t \ge 2^{d^{1/3}}$.
By Lemma~\ref{lemma:cubeExp} and the fact that the average degree of $H$ is at most $6k$, the edge expansion of $H$ is at most $24k\log d/\log t$. Thus, to prove our claim it suffices to show that $\sqrt{72k\log d} \le \log\log t$.
This indeed follows from the theorem's assumption that $k \le \log d/648$, since it implies that
$$\sqrt{72k\log d} \le (\log d)/3 = \log(d^{1/3}) \le \log\log t \;.$$

Having established that every subgraph of $G$ on at least $t_0$ vertices has small edge expansion, we now wish to show that every such graph
has a small separator. So let $H$ be an arbitrary $t$-vertex subgraph of $G$ with $t \ge t_0$.
We apply Lemma~\ref{lemma:boost} on the graph $H$ with the function $f(x)=\frac13(1/\log x)(\log\log x)^2$, noting that, as required by the lemma, every $t'$-vertex subgraph of $H$ with $t' \ge n/3$ $(\ge t_0)$ has expansion at most $f(t')$, 
and $f:[n/3,n]\to\R$ is decreasing (here we use the fact that, as can be easily checked, $f(t)$ is decreasing for $t \ge 256$, and that $n/3 \ge 256$).
We deduce that $\sep(H) \le (2/3)t f(t/3) \le (t/\log t)(\log\log t)^2$,
as required by item $(iii)$ in Theorem~\ref{theo:main}. This completes the proof of the theorem for the case $n=2^d$.

Finally, let us consider the case of an arbitrary $n$, that is, not necessarily a power of $2$. In this case we set $d=\ceil{\log n}$. Since $n>2^d/2$, a random $n$-vertex subgraph of $Q_{d,2k}$ will have, with positive probability, average degree at least $k$. Let $G$ be such a graph.
Then the maximum degree of $G$ is still at most $6k$, its girth is at least $\log n/(6k)^2$, and it is easy to see that every $t$-vertex subgraph $H$ of $G$
with $t\ge \log n/(6k)^2$ satisfies $\sep(H) \le (t/\log t)(\log\log t)^2$. This completes the proof.
\end{proof}

\section{Proof of Theorem \ref{theo:SSE}}\label{sec:proofSSE}

In the first part of the proof we will show that for every $t$-vertex subgraph $H$ of $G$ with $t \ge \frac12\log n/(6k)^2$, and every $1/t\le\mu\le 2/3$, there is a subset $A \sub V(H)$ of size $\mu t/3 \le \s{A} \le \mu t$ satisfying
\begin{equation}\label{eq:SSEclaim}
\phi_H(A) \le \frac{\log (1/\mu)}{\log t}\cdot(4\log\log t)^2 \;.
\end{equation}
In the second part of the proof we will show that if $t\ge \log n/(6k)^2$ (as guaranteed by the statement) and $1/t\le\mu\le 2/3$ then there are in fact $1/4\mu$ mutually disjoint subsets of the same size and edge expansion only slightly larger than in~(\ref{eq:SSEclaim}).

Before we begin, we collect some useful facts.
One can easily verify that the theorem's assumption on $k$ implies that $t\ge\sqrt{\log n}$ $(\ge 2^{16})$,\footnote{$72k^2 \le (\log\log\log n)^2 \le \sqrt{\log n}$; for the last inequality, note the statement of the theorem implies $\log n \ge 2^{2^{2\cdot 648}}$.} and therefore $\D(H) \le \frac12\log\log\log n \le \log\log t$.
Notice we may assume $\mu \ge 1/\sqrt{t}$, as otherwise~(\ref{eq:SSEclaim}) trivially holds by taking $A$ to be any single vertex, in which case $\phi(A) \le \D(H) \le \log\log t \le (\log (1/\mu)/\log t)(4\log\log t)^2$, as needed.
For the rest of the proof we set
$$f(x)=(1/\log x)\cdot(\log\log x)^2 \;.$$
One can check that $f(x)$ is decreasing for $x\ge 256$.
Throughout the proof, unless otherwise mentioned, we write $\phi(\cdot)=\phi_H(\cdot)$ for the edge expansion in $H$, and $\partial(\cdot)=\partial_H(\cdot)$ for the edge boundary in $H$.

In order to obtain a subset $A \subseteq V(H)$ satisfying~(\ref{eq:SSEclaim}) we will iteratively find smaller and smaller subsets, such that in each step the edge expansion (in $H$) of our subset does not increase by much.
Formally, we claim that for any set $S \sub V(H)$ with $\s{S}\ge \sqrt{t}$ $(\ge 256)$ there is a subset $S_1 \sub S$ of size 
$\frac13\s{S} \le \s{S_1} \le \frac23\s{S}$
such that $\phi(S_1) \le \phi(S)+2f(\s{S})$. To see this, consider an edge separator in the induced subgraph $H[S]$. Namely, let $S=S'\cup S''$ be a partition of $S$ with $\s{S'},\s{S''} \le \frac23\s{S}$ such that the number of edges between $S'$ and $S''$ in $H[S]$, and thus in $H$, is $\sep(H[S])$.
Consider now the edge boundary of $S'$ and $S''$ in $H$.
We have 
$$\s{\partial(S')} + \s{\partial(S'')} = \s{\partial(S)} + 2\sep(H[S]) \;.$$
Moving to edge expansion, observe that this means that
$$\min\{\phi(S'),\,\phi(S'')\} \le \frac{\s{\partial(S')} + \s{\partial(S'')}}{\s{S'}+\s{S''}} =
\frac{\s{\partial(S)} + 2\sep(H[S])}{\s{S}} \le
\phi(S) + 2f(\s{S}) \;,$$
where in the left inequality we used the elementary inequality $(a+b)/(c+d) \ge \min\{a/c,\,b/d\}$,
and in the right inequality we used the fact that $\sep(H[S]) \le \s{S}f(\s{S})$. This
last inequality follows from item~(iii) of Theorem~\ref{theo:main} when $\s{S} \ge \log n/(6k)^2$, while when $\s{S} < \log n/(6k)^2$
it follows from item~(ii) of Theorem~\ref{theo:main}, Fact~\ref{fact:treeSep} and the fact that $f$ is decreasing, as these facts together mean
that $\sep(H[S]) \le \D(H) \le \log\log t \le \sqrt{t} f(t) \le \s{S}f(\s{S})$. Altogether we see that we can take either $S'$ or $S''$ as the set $S_1$.
Applying this claim with $S=S_0=V(H)$ yields a subset $S_1 \sub S_0$ of size $\frac13 \s{S_0} \le \s{S_1} \le \frac23 \s{S_0}$ satisfying
$\phi(S_1) \le \phi(S_0)+2f(\s{S_0}) = 2f(\s{S_0})$, where here we use the fact that $S_0=V(H)$ so $\phi(S_0)=0$.
Having picked $S_i$, for some $i \geq 1$, we can then pick a new set $S_{i+1} \subseteq S_i$ satisfying
$\frac13 \s{S_i} \le \s{S_{i+1}} \le \frac23 \s{S_i}$ and $\phi(S_{i+1}) \le \phi(S_i)+2f(\s{S_i})$.
We stop the process once we obtain a subset $A \sub V(H)$ of size $\s{A} \le \mu t$. Note that this means that $\s{A} \ge \mu t/3$.
Let $V(H)=S_0 \supseteq S_1 \supseteq \cdots \supseteq S_s=A$ denote the subsets we thus obtain. We get from the above observations that
$$\phi(A) \le \sum_{i=0}^{s-1} 2f(\s{S_i}) \le s \cdot 2f(\mu t) \;,$$
where in the second inequality we used the fact that $f(x)$ is decreasing.
Since the number $s$ of iterations is at most $\lceil \log_{3/2}(1/\mu) \rceil \leq 2\log_{3/2}(1/\mu)$, we get
$$\phi(A) \le 2\log_{3/2}(1/\mu) \cdot 2f(\sqrt{t}) \le \frac{\log(1/\mu)}{\log t}\cdot 16(\log\log t)^2 \;,$$
where in the first inequality we used the assumption that $\mu t \ge \sqrt{t}$.
This completes the first part of the proof.


Having found a single non-expanding subset $A$ satisfying~(\ref{eq:SSEclaim}), 
we proceed to find at least $1/4\mu$ disjoint subsets of size in $[\mu t/3,\mu t]$
having small expansion, as required by the theorem's statement.
So let $H=(V,E)$ be a $t$-vertex subgraph with $t \ge \log n/(6k)^2$ and suppose $1/t\le \mu \le 2/3$.
Since $\s{V} \ge \frac12\log n/(6k)^2$, we can find a subset $A_1 \sub V$ satisfying $\phi_H(A_1) \le \varphi(\mu,t)$
where $\varphi(\mu,t)$ is the right hand side of (\ref{eq:SSEclaim}).
Now apply the same argument on the induced subgraph $H[V \sm A_1]$ in order to find a subset $A_2 \sub V\sm A_1$ 
of size in $[\mu t/3,\mu t]$ that satisfies\footnote{A subset $A_2 \sub V \sm A_1$ of size at most $\mu t$ is therefore of size at most $\mu'\s{V \sm A_1}$ with $\mu' \ge \mu$. 
Thus, we can find a subset $A_2$ satisfying $\phi_{H[V \sm A_1]} (A_2) \le \varphi(\mu',t) \leq \varphi(\mu,t)$ where we use the
monotonicity of $\varphi(\mu,t)$ with respect to $\mu$.}
$\phi_{H[V \sm A_1]} (A_2) \leq \varphi(\mu,t)$.
Note that we can pick sets $A_i \sub V \sm \bigcup_{j=1}^{i-1} A_j$ at least $1/2\mu$ times (assume henceforth that $1/2\mu$ is an integer), since as long as we pick fewer sets we are left with at least $t/2 \geq \frac12\log n/(6k)^2$ vertices, so we can still pick another set.
We therefore obtain at least $1/2\mu$ 
mutually disjoint subsets $A_1,\ldots,A_{1/2\mu} \sub V$ of size in $[\mu t/3,\mu t]$ each, satisfying
\begin{equation}\label{eq:final}
\phi_{H[V \sm \bigcup_{j=1}^{i-1} A_j]} (A_i) \le \varphi(\mu,t)
\end{equation}
for every $1\le i\le 1/2\mu$. To finish the proof, we claim that
at least half of the $1/2\mu$ subsets $A_i$ have small expansion in $H$ itself. To see this, note that
\begin{equation}\label{eq:final-1}
\sum_{i=1}^{1/2\mu} \s{\partial_H(A_i)}
\le 2\sum_{i=1}^{1/2\mu} \big\lvert \partial_{H[V \sm \bigcup_{j=1}^{i-1} A_j]} (A_i) \big\rvert = 2\sum_{i=1}^{1/2\mu} \s{A_i}\cdot \phi_{H[V \sm \bigcup_{j=1}^{i-1} A_j]} (A_i) \le t\varphi(\mu,t) \;,
\end{equation}
where in the first inequality we use the fact that each edge is counted (at most) twice,
and in the second inequality we use~(\ref{eq:final}) and the fact that each $A_i$ is of size at most $\mu t$.
Dividing both sides of~(\ref{eq:final-1}) by $\mu t/3$ (which is an upper bound on the size of the sets $A_i$) we get that
$\sum_{i=1}^{1/2\mu} \phi_H(A_i) \leq 3\varphi(\mu,t)/\mu$. Hence, by averaging, at least $1/4\mu$ of the sets
$A_i$ satisfy
$$
\phi_H(A_i) \leq 12\varphi(\mu,t) \le (\log (1/\mu)/\log t)(14\log\log t)^2\;,
$$
thus completing the proof.


\section{Missing Proofs}\label{sec:add}

\subsection{Proof of Lemma~\ref{lemma:girthE}}

\begin{proof}[Proof of Lemma~\ref{lemma:girthE}]
Let $k$ denote the average degree of our graph $G=(V,E)$. Since $G$ is connected and not a tree we have $k \geq 2$.
As is well known, $G$ must have a subgraph of minimum degree at least $\floor{k/2}+1 \ge 2$.
In particular, this subgraph contains a cycle. Let $e=uv$ be any edge in such a cycle, and
let $G'$ be obtained from $G$ by removing $e$. Note that our choice of $e$ guarantees that $\deg_{G'}(u),\deg_{G'}(v) \geq \floor{k/2}$.
We claim that the edge expansion of $G'$ is at least $\a/2$. Indeed, let $S\sub V$ with $e\in\partial_G(S)$. Then $\s{\partial_G(S)} \ge 2$, as otherwise $S$ forms a connected component of $G'$, contradicting the fact that $e$ is not a cut edge in $G$.
Therefore, removing a single edge can decrease the edge expansion of $G$ by at most a factor of $2$.

For a vertex $x\in V$ and a non-negative integer $r$, write $B_x(r)=\{y\,:\,\d_{G'}(x,y)\le r\}$.  
Let $m_x(r)$ denote the number of edges spanned by $B_x(r)$.
Since $G'$ has edge expansion at least $\a/2$, we have $m_x(r+1) \ge (1+\a/2)m_x(r)$ provided $\s{B_x(r)} \le n/2$.
Put $R=(2/\a +1)(\ln(n)+1)$. We claim that $\s{B_v(R)} > n/2$. Indeed, recalling $m_v(1)=\deg_{G'}(v)\ge \floor{k/2}$, we get that if $\s{B_v(R)} \leq n/2$ then
$$m_v(R+1) \ge (1+\a/2)^R\cdot \floor{k/2} \ge e^{\ln(n)+1}\cdot \floor{k/2} \ge  nk/2 \;,$$
which is strictly greater than the number of edges of $G'$ -- a contradiction.
Since we also have $\deg_{G'}(u)\ge \floor{k/2}$ we deduce that $B_u(R)\cap B_v(R)\neq \emptyset$. This means that $\d_{G'}(u,v) \le 2R$. Since $u$ and $v$ are adjacent in $G$ but not in $G'$, we conclude that $G$ has a cycle of length at most $2R+1 = 2(2/\a +1)(\ln(n)+1) + 1$.
\end{proof}

\subsection{Proof of Lemma~\ref{lemma:boost}}

\begin{proof}[Proof of Lemma~\ref{lemma:boost}]
We iteratively remove subsets $S_1,\ldots,S_k \sub V(G)$ as follows.
For the $\ith$ step, consider the graph $G_i := G-\bigcup_{j=1}^{i-1} S_j$ and let $S_i \sub V(G_i)$ have size $\s{S_i} \le \s{V(G_i)}/2$ and expansion in $G_i$ at most $f(\s{V(G_i)})$.
We stop removing subsets once the set $S:=\bigcup_{j=1}^k S_j$ is of size at least $n/3$, and thus
$\s{S} \le \sum_{j=1}^{k-1} \s{S_j} + \s{V(G_k)}/2 = (n+\sum_{j=1}^{k-1} \s{S_j})/2 \le 2n/3 \;.$
Now, observe that every edge in the edge boundary $\partial_G(S)$ is a member of some edge boundary $\partial_{G_j}(S_j)$. Hence, the expansion of $S$ in $G$ is
$$\frac{\s{\partial_G(S)}}{\s{S}} \le \frac{\sum_{j=1}^k \s{\partial_{G_j}(S_j)}}{\sum_{j=1}^k \s{S_j}} \le
\max_{1\le j\le k} \frac{\s{\partial_{G_j}(S_j)}}{\s{S_j}} \le \max_{1\le j\le k} f(\s{V(G_j)}) \le f(n/3) \;,$$
where in the second inequality we used the inequality $(\sum_{j=1}^k x_j)/(\sum_{j=1}^k y_j) \le \max_{1 \le j \le k} (x_j/y_j)$ for non-negative reals $x_j,y_j$,\footnote{We note that while the simple argument here suffices to bound the expansion of $\bigcup S_i$, bounding the expansion of each individual $S_i$ -- which is done in the second part of the proof of Theorem~\ref{theo:SSE} -- is trickier.}
and in the last inequality we used the fact that $\s{V(G_j)} \ge \s{V(G_k)} \ge n-\s{S} \ge n/3$ together with the fact that $f(t)$ is decreasing for $t \ge n/3$.
It follows that $\s{\partial_G(S)} \le \s{S}f(n/3) \le (2/3)nf(n/3)$, completing the proof.
\end{proof}

\subsection{Proof of Lemma \ref{lemma:Cheeger}}

\begin{proof}[Proof of Lemma~\ref{lemma:Cheeger}]
Put $L=dI-A$ where $A$ is the adjacency matrix of $G=(V,E)$.
Note that the $\kth$ eigenvalue of $L$ is $d-\lambda_k$. By the Courant-Fischer well-known min-max theorem, we have
\begin{equation}\label{eq:MinMax}
d-\lambda_k = \min_{w^1,\ldots,w^k} \, \max_{0\neq x\in\spanV\{w^1,\ldots,w^k\}} R_L(x) \;,
\end{equation}
where the minimum is over all collections of $k$ mutually orthogonal, nonzero vectors in $\R^{\s{V}}$, and $R_L(x)=x^t L x/x^t x$ is the Rayleigh quotient.
Let $S_1,\ldots,S_k \sub V$ be mutually disjoint, non-empty subsets, and denote by $\mathbf{1}_{S_i}$ the characteristic $\{0,1\}$-vector of $S_i$.
As it is easy to see that $x^tLy=\sum_{\{u,v\}\in E} (x_u-x_v)(y_u-y_v)$, we deduce that $\mathbf{1}_{S_i}^t L \mathbf{1}_{S_i}=\s{\partial(S_i)}$, and that for
$i \neq j$ we have $\mathbf{1}_{S_i}^t L \mathbf{1}_{S_j}=-e(S_i,S_j)$ where $e(S_i,S_j)$ is the number of edges between $S_i$ and $S_j$.
Consider now any vector $x \in \spanV\{\mathbf{1}_{S_1},\ldots,\mathbf{1}_{S_k}\}$ and write $x = \sum_{i=1}^k c_i \mathbf{1}_{S_i}$. Then
\begin{align*}
x^t L x &= \sum_{i,j=1}^k c_i c_j \cdot \mathbf{1}_{S_i}^t L \mathbf{1}_{S_j} = \sum_{i=1}^k c_i^2 \s{\partial(S_i)} - \sum_{i\neq j} c_ic_j e(S_i,S_j)
\le \sum_{i=1}^k c_i^2 \s{\partial(S_i)} + \sum_{i\neq j} \frac{c_i^2+c_j^2}{2} e(S_i,S_j)\\
&= \sum_{i=1}^k c_i^2 \s{\partial(S_i)} + \sum_{i=1}^k c_i^2 \sum_{j\neq i} e(S_i,S_j) \le 2\sum_{i=1}^k c_i^2 \s{\partial(S_i)} \;.
\end{align*}
The proof now follows since~(\ref{eq:MinMax}) implies
$$d-\lambda_k \le \max_{0\neq x\in\spanV\{\mathbf{1}_{S_1},\ldots,\mathbf{1}_{S_k}\}} R_L(x) \le \frac{2\sum_{i=1}^k c_i^2 \s{\partial(S_i)}}{\sum_{i=1}^k c_i^2\s{S_i}} \le 2\max_{1\le i\le k} \phi_G(S_i) \;.$$
\end{proof}

\section{Concluding Remarks and Open Problems}\label{sec:conc}

\begin{itemize}

\item As we mentioned in Section \ref{sec:intro}, we can replace the $(\log\log t)^2$ term in~(\ref{eq:expansion}) with $(\log\log t)^{1+o(1)}$.
It would be interesting to know whether the $\log\log t$ factors in~(\ref{eq:expansion}) (and thus also in~(\ref{theo:SSE})) can be removed.
One way to obtain this is to improve Lemma~\ref{lemma:cubeExp}. In this regard, we conjecture the following ``higher-order'' isoperimetric property of the hypercube: Among all subgraphs of the $d$-cube on $2^i$ vertices, for every $i\in\N$ at least $d^{1/3}$ say, the subcubes have the largest normalized\footnote{I.e., edge expansion divided by the subgraph's average degree.} edge expansion. (One can also come up with an analogous isoperimetric conjecture for vertex expanders in the hypercube.)
Since a subcube on $t=2^i$ vertices has normalized edge expansion $1/i=1/\log t$,
proving the above conjecture would mean that the additional $\log d$ factor in Lemma~\ref{lemma:cubeExp} is unnecessary for $t \ge 2^{d^{1/3}}$ (note that without the lower bound on $t$, Lemma~\ref{lemma:cubeExp} is tight as witnessed by $K_{1,d}$).
By the proof of Theorem~\ref{theo:main}, this would mean that the graphs $G_{n,k}$ are such that the $\log\log t$ terms in~(\ref{eq:expansion}) and~(\ref{theo:SSE}) can in fact be replaced by $O(k)$. Note that this is best possible, as witnessed by the subgraph $G_{n,k}$ itself.
We note that this would essentially close the gaps in the applications mentioned in Section~\ref{sec:app}, as $k$ is chosen there to be either constant or grow arbitrarily small with $n$.
For example, this would imply a truly polynomial lower bound on the threshold rank in Corollary~\ref{coro:ABS-R}, thus proving that the eigenspace enumeration approach of~\cite{AroraBaSt10} cannot solve $\UG_q(\d)$ in time $\exp(n^{o(1)})$ for constant $\d,q$.

\item It is not hard to see that if every $t$-vertex subgraph $H$ of a graph $G$ satisfies $\sep(H) \le (t/\log t)(\log\log t)^c$ then $G$ must have at most $O(n(\log\log n)^{c'})$ edges.
(Heuristically speaking, the upper bound on the number of edges admits a recursion of the form $f(n) \le 2f(n/2)+(n/\log n)(\log\log n)^c$).
We note that one can construct a graph $G$ satisfying this requirement using arguments similar to those we used to prove Theorem~\ref{theo:main}, and in this case, the resulting graph would have $\Omega(n(\log\log n)^{c''})$ edges. (To be more precise, $c'=c+1$ and $c''=c-1$.)
This means that the graph we construct in Theorem~\ref{theo:main} has the maximum possible number of edges, up to $\log\log n$ factors.

The situation for vertex separators, as in Theorem \ref{theo:LRT}, is not so clear, so it would be interesting to understand the maximum number of edges of $n$-vertex graphs in a hereditary family with vertex separators of size at most $(n/\log n)(\log\log n)^c$.
Although we construct such a family $\F$ that contains $n$-vertex graphs with $\Omega(n(\log\log n)^{c''})$ edges, we cannot rule out the possibility of improving this to (say) $n\log n$.

\item Let us mention two problems that seem to be related to the types of problems studied here
but that (unfortunately) we cannot resolve. The first is a nice problem of G. Kalai. Suppose $\F$ is a hereditary family of graphs and that every graph in the family has a vertex separator of size $n/f(n)$. Then, how fast should $f(n)$ grow so that $\F$ has only $2^{O(n)}$ non-isomorphic graphs on $n$ vertices? It was shown in~\cite{Kalai,Kalai-2} that $f(n) \geq \log^{2+\epsilon}n$ suffices while $f(n) \leq \log^{1-\epsilon}n$ does not.
Closing the gap between these bounds is still open.

The second problem is related to a graph decomposition result from a work of P\v{a}tra\c{s}cu and Thorup \cite{PatrascuTh07}.
They proved that the edge set of every graph can be decomposed into $b=O(\log n)$ subsets $E_1,\ldots,E_b$ so that for every
$1 \leq i \leq b$ the graph spanned by $E_1,\ldots,E_i$ has ``edge expansion'' $1/\log n$, where the notion of edge expansion used
here is slightly different from the (standard) one we used throughout this paper.
It would be interesting to decide whether the parameters in the construction
of \cite{PatrascuTh07} are optimal.

\end{itemize}

\noindent \textbf{Acknowledgment:} We are grateful to Alex Samorodnitsky for very helpful discussions related to this work.




\end{document}